\documentclass[a4paper,12pt,reqno]{amsart}
\usepackage{eurosym}
\usepackage{amsfonts}
\usepackage{amsmath,amsthm,amssymb}
\usepackage{graphicx, color}

\nonstopmode \numberwithin{equation}{section}

\newtheorem{theorem}{Theorem}
\newtheorem{corollary}{Corollary}[section]

\allowdisplaybreaks

\begin{document}
\title[A New Class of Integrals Involving Extended Hypergeometric Function]{ A New Class of Integrals Involving Extended Hypergeometric Function}

\author[G. Rahman, A. Ghaffar, K.S. Nisar and S. Mubeen ]{G. Rahman, A. Ghaffar, K.S. Nisar* and S. Mubeen}

\address{Gauhar Rahman:    Department of Mathematics, International Islamic  University, Islamabad, Pakistan}
\email{gauhar55uom@gmail.com}

\address{A. Ghaffar:    Department of Mathematical Science, BUITEMS Quetta, Pakistan}
\email{abdulghaffar.jaffar@gmail.com}

\address{Kottakkaran Sooppy  Nisar:    Department of Mathematics, College of Arts and Science-Wadi Al-dawaser, 11991,
Prince Sattam bin Abdulaziz University, Alkharj, Kingdom of Saudi Arabia}
\email{ksnisar1@gmail.com, n.sooppy@psau.edu.sa}

\address{Shahid Mubeen: Department of Mathematics, University of Sargodha,   Sargodha, Pakistan}
\email{smjhanda@gmail.com}


\thanks{$^{*}$ Corresponding author}


\begin{abstract}
Our purpose in this present paper is to investigate generalized integration formulas containing the extended generalized hypergeometric function and obtained results are expressed in terms of extended hypergeometric function. Certain special cases of the main results presented here are also pointed out for the extended Gauss' hypergeometric and confluent hypergeometric functions.
\end{abstract}

\maketitle

\section{introduction}
In many areas of applied mathematics, various types of special functions become essential tools for scientists and engineers.
The continuous development of mathematical physics, probability theory and other areas has led to new classes of special functions and their extensions and generalizations. The study of one-variable hypergeometric functions appear in the work of Euler, Gauss, Riemann, and Kummer and the integral representations were studied by Barnes and Mellin. The special properties of one variable hypergeometric function were studied by  Schwarz and Goursat. For more details about the recent works in the field of dynamical systems theory, stochastic systems, non-equilibrium statistical mechanics and quantum mechanics, the readers may refer to the recent work of the researchers \cite{Gorenflo,Gorenflo1, Nisar,Nisar1,Nisar2,Podlubny,Prabhakar} and the references cited therein.

Throughout this paper, we denote by $\mathbb{N}, \mathbb{Z}^-$ and $\mathbb{C}$ the sets of positive integers, negative integers and complex numbers,
respectively, and also $\mathbb{N}_0:= \mathbb{N} \cup\{0\}$ and $\mathbb{Z}_0^-:=\mathbb{Z}^-\cup\{0\}$.

The extended generalized hypergeometric function is defined by Srivastava {\it et al.} \cite[p.487, Eq.(15)]{Srivastava}:
\begin{eqnarray}\label{1}
_{r}F_{s}\left[\begin{array}{c l}
              (a_{1},p),a_{2},\cdots, a_{r};\quad \quad \quad\\\left.
             \quad \quad \quad\quad \quad \quad\quad \quad \quad\quad z \right.\\\left.
              b_{1},\cdots, b_{s};\quad\quad\quad\,\right.
           \end{array}\right]=\sum_{n=0}^{\infty}\frac{(a_{1};p)_{n}(a_{2})_{n}\cdots (a_{r})_{n}}{(b_{1})_{n}\cdots (b_{s})_{n}}\frac{z^{n}}{n!}
\end{eqnarray}
where, in terms of generalized Pochammer symbol $(\mu;p)_{v}$\cite[p.485, Eq.(8)]{Srivastava}:
\[
    (\mu;p)_{v}= 
\begin{cases}
    \frac{\Gamma_{p}(\mu+v)}{\Gamma(v)},& (\Re(p)>0, \mu,v \in \mathbb{C})\\
    (\mu)_{v},              & (p=0, \mu, v\in \mathbb{C})
\end{cases}
\]

Here $\Gamma_{p}(z)$ is the generalized gamma function introduced by Chaudhry and Zubair \cite[p. 9, Eq.(1.66)]{CZ} as follows:

\[
    {\Gamma}_{p}(z)= 
\begin{cases}
    \int_{0}^{\infty}{t^{z-1}exp\left(-p-\frac{p}{t}\right)dt},& (\Re(p)>0, z \in \mathbb{C})\\
    {\Gamma}(z),              & (p=0, \Re(z)>0)
\end{cases}
\]

The corresponding extensions of Gauss's hypergeometric and confluent hypergeometric functions are as follows:
\begin{eqnarray}\label{2}
_{2}F_{1}\left[\begin{array}{c l}
              (a,p),\beta;\quad \quad\\\left.
               \quad \quad\quad \quad \quad\quad z \right.\\\left.
              \gamma;\quad\,\right.
           \end{array}\right]=\sum_{n=0}^{\infty}\frac{(a,p)_{n}\beta_{n}z^n}{(\gamma_{n}n!},
\end{eqnarray}
and
\begin{eqnarray}\label{3}
_{1}F_{1}\left[\begin{array}{c l}
              (a,p);\quad \quad\\\left.
                \qquad \qquad\qquad z \right.\\\left.
              \gamma;\,\quad\right.
           \end{array}\right]=\sum_{n=0}^{\infty}\frac{(a,p)_{n}z^n}{(\beta_{1})_{n}(\gamma_{n}n!},
\end{eqnarray}
where $a,\gamma\in \mathbb{C}$.

In this paper, we derive two new integral formulas involving the generalized hypergeometric function (\ref{1}). Further we
give corollaries as special cases for the extended Gauss' hypergeometric and confluent hypergeometric functions. For
the present investigation, we need the following result of  Oberhettinger \cite{Oberhettinger}.
\begin{eqnarray}\label{5}
\int_0^\infty z^{\alpha-1} (z+b+\sqrt{z^{2}+2bz})^{-\beta} dz=2\beta b^{-\beta}\left(\frac{b}{2}\right)^{\alpha}\frac{\Gamma(2\alpha)
\Gamma(\beta-\alpha)}{\Gamma(1+\alpha+\beta)}.
\end{eqnarray}
where $0<\Re(\alpha) <\Re(\beta) $.
For various other investigations involving certain special functions, interested reader may be referred to several recent papers on the subject (see, for example, \cite{
Baleanu,Choi,Choi1,Srivastava1,Srivastava3,Srivastava4} and the references cited in each of these papers.

\section{Main Result}
In this section, the generalized integral formulas involving the extended generalized hypergeometric function  defined in (\ref{1} are established here by inserting with the suitable argument in the integrand of (\ref{5}) and we express the obtained result in terms of an extended Wright-type hypergeometric function.
\begin{theorem}\label{t1}
Let $a_{1},a_{2},...,a_{r},\delta \in \mathbb{C};$ and $\beta_{1},\beta_{2},...,\beta_{s} \in \mathbb{C}/\mathbb{Z}_{0}^{-};$ with $ \Re(p)>0;
 \Re(\mu)>\Re(\delta)>0; p\geq0$ and $z>0$. Then the following formula holds true:
\begin{eqnarray*}\label{2.1}
\int_0^\infty z^{\delta-1} (z+b+\sqrt{z^{2}+2bz})^{-\mu} {_{r}F_{s}}\left[\begin{array}{c l}
              (a_{1},p),a_{2},...,a_{r};\quad \quad \quad\quad\quad  \\\left.
               \quad\quad \quad \quad\quad \quad \frac{y}{(z+b+\sqrt{z^{2}+2bz})} \right.\\\left.
              \beta_{1},\beta_{2},...,\beta_{s};\quad \quad \quad\quad \quad \quad \quad \,\right.
           \end{array}\right]dz
\end{eqnarray*}
\begin{eqnarray}\label{2.2}
&=&\frac{\Gamma(2\delta)b^{\delta-\mu}\Gamma(\mu+1)\Gamma(\mu-\delta)}{2^{\delta-1}\Gamma(\mu)\Gamma(1+\delta+\mu)}\notag\\
&\times&{}{_{r+2}F_{s+2}}\left[\begin{array}{c l}
              (a_{1},p),a_{2},...,a_{r}, \mu+1,\mu-\delta; \\\left.
                \quad\quad\quad \quad\quad \quad \quad \quad\quad \quad \quad \quad\quad \quad \quad\quad \quad |\frac{y}{b} \right.\\\left.
              \beta_{1},\beta_{2},...,\beta_{s},\mu,\mu+\delta+1;\, \,\right.
           \end{array}\right].
\end{eqnarray}
\end{theorem}
\begin{proof}
Let $S_1$ be the left-hand side of (\ref{2.1}), and applying
\begin{eqnarray}\label{2.3}
_{r}F_{s}\left[\begin{array}{c l}
              (a_{1},p),(a_{2}),...,(a_{r});\quad \quad \quad\\\left.
             \quad \quad \quad\quad \quad \quad\quad \quad \quad\quad;z \right.\\\left.
              (\beta_{1}),(\beta_{2}),...,(\beta_{s});\quad\quad\quad\,\right.
           \end{array}\right]\nonumber=\sum_{n=0}^{\infty}\frac{(a_{1},p)_{n}(a_{2})_{n}...(a_{r})_{n}z^n}{(\beta_{1})_{n}(\beta_{2})_{n}...(\beta_{s})_{n}n!},
\end{eqnarray}
where $a_i,\beta_i\in \mathbb{C} ; i = 1, 2, ..., p; j=1,2,...,q $and $\beta_j \neq 0,-1,-2,... $ and $(z)_n$ is
the Pochhammer symbols.\\
To the integrand (\ref{2.1}), we have
\begin{eqnarray*}\label{2.4}
S_1=\int_0^\infty z^{\delta-1} (z+b+\sqrt{z^{2}+2bz})^{-\mu}\sum_{n=0}^{\infty}\frac{(a_{1},p)_{n}(a_{2})_{n}...(a_{r})_{n}}
{(\beta_{1})_{n}(\beta_{2})_{n}...(\beta_{s})_{n}n!}\left(\frac{y}{z+b+\sqrt{z^{2}+2bz}}\right)^{n}dz.
\end{eqnarray*}
By interchanging the order of integration and summation, which is verified by the uniform convergence of the series under the given assumption of theorem (\ref{2.1}), we have
\begin{eqnarray*}\label{2.5}
S_1=\int_0^\infty z^{\delta-1} (z+b+\sqrt{z^{2}+2bz})^{-(\mu+n)}\sum_{n=0}^{\infty}\frac{(a_{1},p)_{n}(a_{2})_{n}...(a_{r})_{n}y^n}
{(\beta_{1})_{n}(\beta_{2})_{n}...(\beta_{s})_{n}n!}dz.
\end{eqnarray*}
By using the (\ref{5}), we get
\begin{eqnarray*}\label{2.6}
S_1=\sum_{n=0}^{\infty}\frac{(a_{1},p)_{n}(a_{2})_{n}...(a_{r})_{n}y^n}
{(\beta_{1})_{n}(\beta_{2})_{n}...(\beta_{s})_{n}n!}2(\mu+n)b^{-(\mu+n)}\left(\frac{b}{2}\right)^{\delta}\frac{\Gamma(2\delta)
\Gamma(\mu+n-\delta)}{\Gamma(1+\delta+\mu+n)}.
\end{eqnarray*}
\begin{eqnarray*}\label{2.7}
S_1=\frac{\Gamma(2\delta)b^{\delta-\mu}}{2^{\delta-1}}\sum_{n=0}^{\infty}\frac{(a_{1},p)_{n}(a_{2})_{n}...(a_{r})_{n}}
{(\beta_{1})_{n}(\beta_{2})_{n}...(\beta_{s})_{n}}\frac{y^{n}}{n!b^{n}}\frac{\Gamma(\mu+n+1)
\Gamma(\mu+n-\delta)}{\Gamma(\mu+n)\Gamma(1+\delta+\mu+n)}.
\end{eqnarray*}
Upon using the (\ref{1}), we obtain
\begin{eqnarray*}
S_1=\frac{\Gamma(2\delta)b^{\delta-\mu}\Gamma(\mu+1)\Gamma(\mu-\delta)}{2^{\delta-1}\Gamma(\mu)\Gamma(1+\delta+\mu)}{_{r+2}F_{s+2}}\left[\begin{array}{c l}
              (a_{1},p),a_{2},...,a_{r}, \mu+1,\mu-\delta; \\\left.
                \quad\quad\quad \quad\quad \quad \quad \quad\quad \quad \quad \quad\quad \quad \quad\quad \quad |\frac{y}{b} \right.\\\left.
              \beta_{1},\beta_{2},...,\beta_{s},\mu,\mu+\delta+1;\, \,\right.
           \end{array}\right].
\end{eqnarray*}
\end{proof}
\begin{theorem}\label{t2}
Let $a_{1},a_{2},...,a_{r},\delta \in \mathbb{C};$ and $\beta_{1},\beta_{2},...,\beta_{s} \in \mathbb{C}/\mathbb{Z}_{0}^{-};$ with $ \Re(p)>0;
 \Re(\mu)>\Re(\delta)>0; p\geq0$ and $z>0$. Then the following formula holds true:
\begin{eqnarray*}\label{2.9}
\int_0^\infty z^{\delta-1} (z+b+\sqrt{z^{2}+2bz})^{-\mu} {_{r}F_{s}}\left[\begin{array}{c l}
              (a_{1},p),a_{2},...,a_{r};\quad \quad \quad\quad\quad  \\\left.
               \qquad\quad \quad \quad\quad \quad \frac{yz}{(z+b+\sqrt{z^{2}+2bz})} \right.\\\left.
              \beta_{1},\beta_{2},...,\beta_{s};\quad \quad \quad\quad \quad \quad \quad \,\right.
           \end{array}\right]dz
\end{eqnarray*}
\begin{eqnarray}\label{2.10}
&=&\frac{\Gamma(\mu-\delta)b^{\delta-\mu}\Gamma(\delta)\Gamma(\delta+1/2)\Gamma(\mu+1)}{2^{\delta-1}\Gamma(\mu)\Gamma(1+\delta+\mu)}\notag\\
&\times&{}{_{r+3}F_{s+3}}\left[\begin{array}{c l}
              (a_{1},p),a_{2},...,a_{r},\mu+1,\delta+1/2; \\\left.
                \quad\quad \quad\quad \quad \quad \quad\quad \quad \quad \quad\quad \quad \quad\quad \quad|\frac{yz}{b} \right.\\\left.
              \beta_{1},\beta_{2},...,\beta_{s},\mu,(\mu+\delta+1)/2,(\mu+\delta+2)/2 ;\, \,\right.
           \end{array}\right].
\end{eqnarray}
\end{theorem}
\begin{proof}
Let $S_2$ be the left-hand side of (\ref{2.10}), and applying
\begin{eqnarray}\label{2.11}
_{r}F_{s}\left[\begin{array}{c l}
              (a_{1},p),(a_{2}),...,(a_{r});\quad \quad \quad\\\left.
             \quad \quad \quad\quad \quad \quad\quad \quad \quad\quad z \right.\\\left.
              (\beta_{1}),(\beta_{2}),...,(\beta_{s});\quad\quad\quad\,\right.
           \end{array}\right]\nonumber=\sum_{n=0}^{\infty}\frac{(a_{1},p)_{n}(a_{2})_{n}...(a_{r})_{n}z^n}{(\beta_{1})_{n}(\beta_{2})_{n}...(\beta_{s})_{n}n!},
\end{eqnarray}
where $a_i,\beta_i\in \mathbb{C} ; i = 1, 2, ..., p; j=1,2,...,q $and $\beta_j \neq 0,-1,-2,... $ and $(z)_n$ is
the Pochhammer symbols.\\
To the integrand (\ref{2.10}), we have
\begin{eqnarray*}\label{2.12}
S_2=\int_0^\infty z^{\delta-1} (z+b+\sqrt{z^{2}+2bz})^{-\mu}\sum_{n=0}^{\infty}\frac{(a_{1},p)_{n}(a_{2})_{n}...(a_{r})_{n}}
{(\beta_{1})_{n}(\beta_{2})_{n}...(\beta_{s})_{n}n!}\left(\frac{yz}{z+b+\sqrt{z^{2}+2bz}}\right)^{n}dz.
\end{eqnarray*}
By interchanging the order of integration and summation, which is verified by the uniform convergence of the series under the given assumption of theorem (\ref{2.10}), we have
\begin{eqnarray*}\label{2.13}
S_2=\int_0^\infty z^{\delta+n-1} (z+b+\sqrt{z^{2}+2bz})^{-(\mu+n)}\sum_{n=0}^{\infty}\frac{(a_{1},p)_{n}(a_{2})_{n}...(a_{r})_{n}y^n}
{(\beta_{1})_{n}(\beta_{2})_{n}...(\beta_{s})_{n}n!}dz.
\end{eqnarray*}
By using the (\ref{5}), we get
\begin{eqnarray*}\label{2.14}
S_2=\sum_{n=0}^{\infty}\frac{(a_{1},p)_{n}(a_{2})_{n}...(a_{r})_{n}y^n}
{(\beta_{1})_{n}(\beta_{2})_{n}...(\beta_{s})_{n}n!}2(\mu+n)b^{-(\mu+n)}\left(\frac{b}{2}\right)^{\delta+n}\frac{\Gamma(2\delta+2n)
\Gamma(\mu-\delta)}{\Gamma(1+\delta+\mu+n)}.
\end{eqnarray*}
\begin{eqnarray*}\label{2.15}
S_2=\frac{\Gamma(\mu-\delta)b^{\delta-\mu}}{2^{\delta-1}}\sum_{n=0}^{\infty}\frac{(a_{1},p)_{n}(a_{2})_{n}...(a_{r})_{n}}
{(\beta_{1})_{n}(\beta_{2})_{n}...(\beta_{s})_{n}}\frac{y^{n}}{n!b^{n}}\frac{\Gamma(2\delta+2n)\Gamma(\mu+n+1)
}{\Gamma(\mu+n)\Gamma(1+\delta+\mu+2n)}.
\end{eqnarray*}
Upon using the (\ref{1}), we obtain
\begin{eqnarray*}
S_2=\frac{\Gamma(\mu-\delta)b^{\delta-\mu}\Gamma(\delta)\Gamma(\delta+1/2)\Gamma(\mu+1)}{2^{\delta-1}\Gamma(\mu)\Gamma(1+\delta+\mu)}\\
\times{}{_{r+3}F_{s+3}}\left[\begin{array}{c l}
              (a_{1},p),a_{2},...,a_{r},\mu+1,\delta+1/2; \\\left.
                \quad\quad \quad\quad \quad \quad \quad\quad \quad \quad \quad\quad \quad \quad\quad \quad|\frac{yz}{b} \right.\\\left.
              \beta_{1},\beta_{2},...,\beta_{s},\mu,(\mu+\delta+1)/2,(\mu+\delta+2)/2 ;\, \,\right.
           \end{array}\right].
\end{eqnarray*}

\end{proof}
\section{Special Cases}
In this section, we present certain special cases of (2:1) and (2:2) as corollaries given below for extended Gauss'
hypergeometric and confluent hypergeometric functions (1:4) and (1:5).
\begin{corollary}\label{t3.1}
Let $a,\beta,\gamma,\delta \in \mathbb{C};$ and $\gamma \in \mathbb{C}/\mathbb{Z}_{0}^{-};$ with $ \Re(p)>0;
 \Re(\mu)>\Re(\delta)>0; p\geq0$ and $z>0$. Then the following formula holds true:
\begin{eqnarray*}\label{2.16}
\int_0^\infty z^{\delta-1} (z+b+\sqrt{z^{2}+2bz})^{-\mu} {_{2}F_{1}}\left[\begin{array}{c l}
              (a,p),\beta;\quad \quad \quad\quad\quad  \\\left.
               \quad\quad \quad \quad\quad \quad \frac{y}{(z+b+\sqrt{z^{2}+2bz})} \right.\\\left.
              \gamma;\quad \quad \quad\quad \quad \quad \quad \,\right.
           \end{array}\right]dz
\end{eqnarray*}
\begin{eqnarray}\label{2.17}
=\frac{\Gamma(2\delta)b^{\delta-\mu}\Gamma(\mu+1)\Gamma(\mu-\delta)}{2^{\delta-1}\Gamma(\mu)\Gamma(1+\delta+\mu)}{_{4}F_{3}}\left[\begin{array}{c l}
              (a,p),\beta,\mu+1,\mu-\delta; \\\left.
               \quad \quad \quad \quad\quad \quad \quad \quad\quad \quad \quad\quad \quad |\frac{y}{b} \right.\\\left.
              \gamma,\mu,\mu+\delta+1;\, \,\right.
           \end{array}\right].
\end{eqnarray}
\end{corollary}
\begin{proof}
Let $S_3$ be the left-hand side of (\ref{2.17}), and applying
\begin{eqnarray}\label{2.18}
_{2}F_{1}\left[\begin{array}{c l}
              (a,p),\beta;\quad \quad \quad\\\left.
              \quad \quad\quad \quad \quad\quad;z \right.\\\left.
              \gamma;\quad\quad\quad\,\right.
           \end{array}\right]\nonumber=\sum_{n=0}^{\infty}\frac{(a,p)_{n}(\beta)_{n}z^n}{(\gamma)_{n}n!},
\end{eqnarray}
where $a,\beta\in \mathbb{C} $and $\gamma \neq 0,-1,-2,... $ and $(z)_n$ is
the Pochhammer symbols.\\
To the integrand (\ref{2.17}), we have
\begin{eqnarray*}\label{2.20}
S_3=\int_0^\infty z^{\delta-1} (z+b+\sqrt{z^{2}+2bz})^{-\mu}\sum_{n=0}^{\infty}\frac{(a,p)_{n}(\beta)_{n}}
{(\gamma)_{n}n!}\left(\frac{y}{z+b+\sqrt{z^{2}+2bz}}\right)^{n}dz.
\end{eqnarray*}
By interchanging the order of integration and summation, which is verified by the uniform convergence of the series under the given assumption of theorem (\ref{2.17}), we have
\begin{eqnarray*}\label{2.21}
S_3=\int_0^\infty z^{\delta-1} (z+b+\sqrt{z^{2}+2bz})^{-(\mu+n)}\sum_{n=0}^{\infty}\frac{(a,p)_{n}(\beta)_{n}y^n}
{(\gamma)_{n}n!}dz.
\end{eqnarray*}
By using the (\ref{5}), we get
\begin{eqnarray*}\label{2.22}
S_3=\sum_{n=0}^{\infty}\frac{(a,p)_{n}(\beta)_{n}y^n}
{(\gamma)_{n}n!}2(\mu+n)b^{-(\mu+n)}\left(\frac{b}{2}\right)^{\delta}\frac{\Gamma(2\delta)
\Gamma(\mu+n-\delta)}{\Gamma(1+\delta+\mu+n)}.
\end{eqnarray*}
\begin{eqnarray*}
S_3=\frac{\Gamma(2\delta)b^{\delta-\mu}}{2^{\delta-1}}\sum_{n=0}^{\infty}\frac{(a,p)_{n}(\beta)_{n}}
{(\gamma)_{n}}\frac{y^{n}}{n!b^{n}}\frac{\Gamma(\mu+n+1)
\Gamma(\mu+n-\delta)}{\Gamma(\mu+n)\Gamma(1+\delta+\mu+n)}.
\end{eqnarray*}
Upon using the (\ref{1}), we obtain
\begin{eqnarray*}\label{2.23}
S_3=\frac{\Gamma(2\delta)b^{\delta-\mu}\Gamma(\mu+1)\Gamma(\mu-\delta)}{2^{\delta-1}\Gamma(\mu)\Gamma(1+\delta+\mu)}{_{4}F_{3}}\left[\begin{array}{c l}
              (a,p),\beta,\mu+1,\mu-\delta; \\\left.
               \quad \quad \quad \quad\quad \quad \quad \quad\quad \quad \quad\quad \quad |\frac{y}{b} \right.\\\left.
              \gamma,\mu,\mu+\delta+1;\, \,\right.
           \end{array}\right].
\end{eqnarray*}

\end{proof}
\begin{corollary}\label{t3.2}
Let $a,\beta,\gamma,\delta \in \mathbb{C};$ and $\gamma\in \mathbb{C}/\mathbb{Z}_{0}^{-};$ with $ \Re(p)>0;
 \Re(\mu)>\Re(\delta)>0; p\geq0$ and $z>0$. Then the following formula holds true:
\begin{eqnarray*}\label{2.24}
\int_0^\infty z^{\delta-1} (z+b+\sqrt{z^{2}+2bz})^{-\mu} {_{2}F_{1}}\left[\begin{array}{c l}
              (a,p),\beta;\quad \quad \quad\quad\quad  \\\left.
                \quad \quad\quad \quad \frac{yz}{(z+b+\sqrt{z^{2}+2bz})} \right.\\\left.
              \gamma;\quad \quad \quad\quad \quad \quad \quad \,\right.
           \end{array}\right]dz
\end{eqnarray*}
\begin{eqnarray}\label{2.25}
=\frac{\Gamma(\mu-\delta)b^{\delta-\mu}\Gamma(\delta)\Gamma(\mu+1)\Gamma(\delta+\frac{1}{2})}{2^{\delta-1}\Gamma(1+\delta+\mu)\Gamma(\mu)}{_{5}F_{4}}\left[\begin{array}{c l}
              (a,p),\beta,\mu+1,\delta,\delta+\frac{1}{2}; \\\left.
                \quad \quad\quad \quad \quad \quad\quad \quad \quad\quad \quad |\frac{yz}{b} \right.\\\left.
              \gamma,\mu,\frac{\mu+\delta+1}{2}, \frac{\mu+\delta+2}{2};\, \,\right.
           \end{array}\right].
\end{eqnarray}
\end{corollary}
\begin{proof}
Let $S_4$ be the left-hand side of (\ref{2.25}), and applying
\begin{eqnarray}\label{2.26}
_{2}F_{1}\left[\begin{array}{c l}
              (a,p),\beta;\quad \quad \quad\\\left.
             \quad \quad\quad \quad \quad\quad z \right.\\\left.
              \gamma;\quad\quad\quad\,\right.
           \end{array}\right]\nonumber=\sum_{n=0}^{\infty}\frac{(a,p)_{n}\beta_{n}z^n}{(\beta_{1})_{n}(\gamma)_{n}n!}
\end{eqnarray}
where $a,\beta\in \mathbb{C} $and $\gamma \neq 0,-1,-2,... $ and $(z)_n$ is
the Pochhammer symbols.\\
To the integrand (\ref{2.25}), we have
\begin{eqnarray*}\label{2.27}
S_4=\int_0^\infty z^{\delta-1} (z+b+\sqrt{z^{2}+2bz})^{-\mu}\sum_{n=0}^{\infty}\frac{(a,p)_{n}(\beta)_{n}}
{(\gamma)_{n}n!}\left(\frac{yz}{z+b+\sqrt{z^{2}+2bz}}\right)^{n}dz.
\end{eqnarray*}
By interchanging the order of integration and summation, which is verified by the uniform convergence of the series under the given assumption of theorem (\ref{2.25}), we have
\begin{eqnarray*}\label{2.28}
S_4=\int_0^\infty z^{\delta+n-1} (z+b+\sqrt{z^{2}+2bz})^{-(\mu+n)}\sum_{n=0}^{\infty}\frac{(a,p)_{n}(\beta)_{n}y^n}
{(\gamma)_{n}n!}dz.
\end{eqnarray*}
By using the (\ref{5}), we get
\begin{eqnarray*}\label{2.29}
S_4=\sum_{n=0}^{\infty}\frac{(a,p)_{n}(\beta)_{n}y^n}
{(\gamma)_{n}n!}2(\mu+n)b^{-(\mu+n)}\left(\frac{b}{2}\right)^{\delta+n}\frac{\Gamma(2\delta+2n)
\Gamma(\mu-\delta)}{\Gamma(1+\delta+\mu+n)}.
\end{eqnarray*}
\begin{eqnarray*}\label{2.30}
S_4=\frac{\Gamma(\mu-\delta)b^{\delta-\mu}}{2^{\delta-1}}\sum_{n=0}^{\infty}\frac{(a,p)_{n}(\beta)_{n}}
{(\gamma)_{n}}\frac{y^{n}}{n!b^{n}}\frac{\Gamma(2\delta+2n)\Gamma(\mu+n+1)
}{\Gamma(\mu+n)\Gamma(1+\delta+\mu+2n)}.
\end{eqnarray*}
Upon using the (\ref{1}), we obtain
\begin{eqnarray*}\label{2.31}
S_4=\frac{\Gamma(\mu-\delta)b^{\delta-\mu}\Gamma(\delta)\Gamma(\mu+1)\Gamma(\delta+\frac{1}{2})}{2^{\delta-1}\Gamma(1+\delta+\mu)\Gamma(\mu)}{_{5}F_{4}}\left[\begin{array}{c l}
              (a,p),\beta,\mu+1,\delta,\delta+\frac{1}{2}; \\\left.
                \quad \quad\quad \quad \quad \quad\quad \quad \quad\quad \quad |\frac{yz}{b} \right.\\\left.
              \gamma,\mu,\frac{\mu+\delta+1}{2}, \frac{\mu+\delta+2}{2};\, \,\right.
           \end{array}\right].
\end{eqnarray*}

\end{proof}
\begin{corollary}\label{t5}
Let $a,\beta,\gamma,\delta \in \mathbb{C};$ and $\gamma \in \mathbb{C}/\mathbb{Z}_{0}^{-};$ with $ \Re(p)>0;
 \Re(\mu)>\Re(\delta)>0; $ and $z>0$. Then the following formula holds true:
\begin{eqnarray*}\label{2.32}
\int_0^\infty z^{\delta-1} (z+b+\sqrt{z^{2}+2bz})^{-\mu} {_{2}F_{1}}\left[\begin{array}{c l}
              a,\beta;\quad \quad \quad\quad\quad \quad\quad \\\left.
               \quad\quad \quad \quad\quad \quad \frac{y}{(z+b+\sqrt{z^{2}+2bz})} \right.\\\left.
              \gamma;\quad \quad \quad\quad \quad \quad \quad \,\right.
           \end{array}\right]dz
\end{eqnarray*}
\begin{eqnarray}\label{2.33}
=\frac{\Gamma(2\delta)b^{\delta-\mu}\Gamma(\mu-\delta)\Gamma(\mu+1)}{2^{\delta-1}\Gamma(\mu)\Gamma(\mu+\delta+1)}{_{4}F_{3}}\left[\begin{array}{c l}
              a,\beta,\mu+1,\mu-\delta; \\\left.
                \quad\quad \quad \quad \quad\quad \quad \quad\quad \quad |\frac{y}{b} \right.\\\left.
              \gamma,\mu,\mu+\delta+1;\, \,\right.
           \end{array}\right].
\end{eqnarray}
\end{corollary}
\begin{corollary}\label{t6}
Let $a,\beta,\gamma,\delta \in \mathbb{C};$ and $\gamma\in \mathbb{C}/\mathbb{Z}_{0}^{-};$ with $ \Re(p)>0;
 \Re(\mu)>\Re(\delta)>0; p\geq0$ and $z>0$. Then the following formula holds true:
\begin{eqnarray*}\label{2.34}
\int_0^\infty z^{\delta-1} (z+b+\sqrt{z^{2}+2bz})^{-\mu} {_{2}F_{1}}\left[\begin{array}{c l}
              a,\beta;\quad \quad \quad\quad\quad  \\\left.
                \quad \quad\quad \quad \frac{yz}{(z+b+\sqrt{z^{2}+2bz})} \right.\\\left.
              \gamma;\quad \quad \quad\quad \quad \quad \quad \,\right.
           \end{array}\right]dz
\end{eqnarray*}
\begin{eqnarray}\label{2.35}
=\frac{\Gamma(\mu-\delta)b^{\delta-\mu}\Gamma(\mu+1)\Gamma(\delta)\Gamma(\delta+\frac{1}{2})}{2^{\delta-1}\Gamma(\mu)\Gamma(\mu+\delta+1)}{_{5}F_{4}}\left[\begin{array}{c l}
              a,\beta,\mu+1,\delta,\delta+\frac{1}{2} ; \\\left.
                \quad \quad\quad \quad \quad \quad\quad \quad \quad\quad \quad |\frac{yz}{b} \right.\\\left.
              \gamma,\mu,\mu+\delta+1,\frac{\mu+\delta+1}{2},\frac{\mu+\delta+2}{2} ;\, \,\right.
           \end{array}\right].
\end{eqnarray}
\end{corollary}

\end{document}